\documentclass{amsart}

\usepackage{amsmath}
\usepackage{amsthm}
\usepackage{amssymb}
\usepackage{amsfonts}
\usepackage{xcolor}
\usepackage{amsaddr}
\usepackage[colorlinks=true,urlcolor=blue,linkcolor=blue,plainpages=false,pdfpagelabels]{hyperref}

\newtheorem{theorem}{Theorem}[section]
\newtheorem*{theorem-non}{Theorem}
\newtheorem{definition}[theorem]{Definition}
\newtheorem{proposition}[theorem]{Proposition}
\newtheorem{lemma}[theorem]{Lemma}

\newtheorem{fact}[theorem]{Remark}

\newcommand{\N}{{\mathbb N}}

\newcommand{\roneq}{\sigma_{1}}
\newcommand{\ronestarq}{\sigma^*_{1}}
\newcommand{\rtwoq}{\sigma_{2}}
\newcommand{\rthreeq}{\sigma_{3}}
\newcommand{\rfourq}{\sigma_{4}}

\newcommand{\rsixq}{\sigma_{5}}

\newcommand{\cmh}{H}

\title{On modified Halpern and Tikhonov-Mann iterations}
\author{\!\!\!Hora\c tiu Cheval${}^{\MakeLowercase{a}}$, Ulrich Kohlenbach${}^{\MakeLowercase{b}}$, Lauren\c{t}iu Leu\c{s}tean${}^{\MakeLowercase{a,c,d}}$}\parindent0pt
\address{${}^a$ LOS, Faculty of Mathematics and Computer Science, University of Bucharest,\newline
${}^b$ Department of Mathematics, Technische Universit\"at Darmstadt,\newline
${}^c$  Simion Stoilow Institute of Mathematics of the Romanian Academy,\newline 
${}^{d}$ Institute for Logic and Data Science,Bucharest,\\[1mm]
Email: \rm horatiu.cheval@unibuc.ro, kohlenbach@mathematik.tu-darmstadt.de, 
laurentiu.leustean@unibuc.ro}

\begin{document}


\begin{abstract}
We show that the asymptotic regularity and the
strong convergence of the modified Halpern iteration due to T.-H. Kim and 
H.-K. Xu and studied further by A. Cuntavenapit and B. Panyanak
and the Tikhonov-Mann iteration introduced by H. Cheval and L. Leu\c{s}tean 
as a generalization of an iteration due to Y. Yao et al. that has recently 
been studied by Bo\c{t} et al. can be reduced to each other in general 
geodesic settings. This, in particular, gives a new proof of the convergence 
result in Bo\c{t} et al. together with a generalization from Hilbert
to CAT(0) spaces. Moreover, quantitative rates of asymptotic regularity and 
metastability due to K. Schade and U. Kohlenbach can be adapted and 
transformed into rates for the Tikhonov-Mann iteration 
corresponding to recent quantitative results on the latter of H. Cheval, 
L. Leu\c{s}tean and B. Dinis, P. Pinto respectively. A transformation in 
the converse direction is also possible. We also obtain 
rates of asymptotic regularity of order $O(1/n)$ for both the modified Halpern 
(and so in particular for the Halpern iteration) 
and the Tikhonov-Mann iteration in a general geodesic setting for a special 
choice of scalars.\\



\end{abstract}

 \maketitle
 
\section{Introduction}
 
We consider in the sequel generalizations of the well-known Mann and Halpern iterations obtained by 
combining them with the so-called Tikhonov regularization terms \cite{Att96,LehMou96}. Although we will in the rest of the paper work in a general geodesic 
setting we first discuss these iterations for simplicity 
in the context of linear normed spaces, where $X$ is a Banach space, $C\subseteq X$ is a convex subset, and $T:C\to C$ is a nonexpansive mapping.

One such generalization is the \textit{Tikhonov-Mann iteration}, defined in \cite{CheLeu22} as follows:
\begin{equation}
x_{n + 1} = (1 - \lambda_n)((1-\beta_n)u+\beta_nx_n) + \lambda_n T\big((1-\beta_n)u+\beta_nx_n\big),
\label{def-TKM-Banach} 
\end{equation}
where $(\lambda_n)_{n \in \N},(\beta_n)_{n \in \N}$ are sequences in $[0, 1]$ and $x_0,u\in C$. 
Obviously, if $\beta_n=1$, then  $(x_n)$ becomes the Mann iteration. For $u=0$ one 
gets a modified Mann iteration, studied in \cite{YaoZhoLio09} and rediscovered in a recent paper \cite{BotCseMei19}.

Another generalization is the \textit{modified Halpern iteration}, introduced in \cite{KimXu05}:
\begin{align}
y_{n + 1} := \gamma_n v + (1 - \gamma_n)(\alpha_n y_n + (1 - \alpha_n)T y_n),
\label{def-MH-Banach}
\end{align}
where $(\gamma_n)_{n \in \N}$ and $(\alpha_n)_{n \in \N}$ are sequences in $[0, 1]$ and $v, y_0 \in C$. \\ 
\cite{KimXu05} showed the strong convergence of $(y_n)$ in uniformly 
smooth Banach spaces under certain conditions on the scalars and assuming that $T$ has a fixed point. Under somewhat 
more liberal conditions, \cite{CunPan11} 
showed the strong convergence in the nonlinear 
setting of CAT(0) spaces. \\[1mm]
In this paper we establish, in a general nonlinear setting, a strong connection between the modified Halpern and the Tikhonov-Mann iteration schemes. 
From this connection it follows that the strong convergence of one scheme 
implies that of the other. In particular, the strong convergence 
theorem of \cite{BotCseMei19} follows from the much older results on the 
modified Halpern iteration and - by a slight modification of the argument 
provided in \cite{CunPan11} 
- also under the exact same conditions on the scalars 
as assumed in \cite{BotCseMei19}. Moreover, quantitative rates of asymptotic 
regularity and metastability (in the sense of Tao 
\cite{Tao(07),Tao(07a)}) for one scheme translate 
into corresponding rates for the other scheme. In 2012, Schade and the 
second author \cite{KohSch12} extracted rates of asymptotic regularity 
for the modified Halpern iteration in a general nonlinear setting and 
of metastability for CAT(0) spaces from the proof given in  
\cite{CunPan11}. We show that with a slight modification of the rate 
of asymptotic regularity one can in that extraction weaken again the conditions 
on the scalars to those used in  \cite{BotCseMei19}. By the aforementioned 
reduction of the Tikonov-Mann iteration to the modified Halpern iteration
this induces corresponding rates of asymptotic regularity in a general 
nonlinear setting and of metastability for CAT(0) spaces for the 
Tikhonov-Mann scheme. A rate of asymptotic regularity in this case has 
recently been extracted in \cite{CheLeu22} directly from 
the proof given in \cite{BotCseMei19} by the first and the 
third author and a rate of metastability has recently been obtained for 
the modified Mann iteration, a special case of the Tikhonov-Mann iteration,
in the case of Hilbert spaces in \cite{DinPin21}. 
In \cite{DinPin22} the authors introduce an alternating 
Halpern-Mann iteration and compute 
rates of metastability for this iteration in the setting of CAT(0) spaces. As
the Tikhonov-Mann iteration $(x_n)$ is a special case of the alternating Halpern-Mann iteration, one gets, as a corollary of
\cite[Theorem~5.1]{DinPin22} rates of metastability for $(x_n)$. However, 
the proof of \cite[Theorem~5.1]{DinPin22} uses a stronger condition on 
the scalars than in our result. \\
Conversely, rates for the Tikhonov-Mann iteration  imply - by the 
connection established in this paper - corresponding rates for the modified 
Halpern iteration. \\ For some special test case for the choice of scalars, 
we for the first time obtain rates of asymptotic regularity of order 
$O(1/n)$ for both iterations which are new even in the linear case.

 \section{$W$-spaces}

Firstly, let us recall some basic notions  from geodesic geometry. 
We refer to \cite{Pap05} for details. Let  $(X,d)$ be a metric space. 
A geodesic path (or simply a geodesic) in $X$ is a function $\gamma:[a,b]\to X$ 
which is distance-preserving, that is $
d(\gamma(s),\gamma(t))=|s-t| \text{~~for all~~} s,t\in [a,b]. $ 
A geodesic segment in $X$ is the image of a geodesic in $X$. If $\gamma:[a,b]\to X$ is 
a geodesic, $\gamma(a)=x$ and $\gamma(b)=y$, we say that 
the geodesic  $\gamma$ joins $x$ and $y$ or that the geodesic  segment $\gamma([a,b])$ 
joins $x$ and $y$.
The metric space $(X,d)$ is (uniquely) geodesic if every two points of $X$ are joined by a (unique) 
geodesic segment. The following useful properties are well-known.

\begin{lemma}\label{geodesic-prop}
Assume that $X$ is a geodesic space.
\begin{enumerate}
\item\label{geodesic-prop-1} Let $x,y\in X$ and $\gamma([a,b])$  be a geodesic segment that joins $x$ and $y$. 
For every $\lambda\in [0,1]$, $z=\gamma((1-\lambda)a+\lambda b)$ 
is the unique point in $\gamma([a,b])$ satisfying $d(z,x)=\lambda d(x,y)$, 
and this unique $z$ satisfies also $d(z,y)=(1-\lambda)d(x,y)$.
\item\label{geodesic-prop-unique} The following are equivalent: 
\begin{enumerate}
\item $X$ is uniquely geodesic.
\item For any $x,y\in X$  and any $\lambda\in [0,1]$ there exists a unique element $z\in X$ such that
$d(x,z) = \lambda d(x,y) \text{~and~}  d(y,z) = (1-\lambda)d(x,y)$.
\end{enumerate}
\end{enumerate}
\end{lemma}

\mbox{}

As in \cite{CheLeu22}, we consider a $W$-space to be a metric space $(X,d)$  together with 
a function $W:X\times X\times [0,1]\to X$. We think of $W(x,y,\lambda)$ as an abstract convex 
combination of the points $x,y\in X$. That is why we shall write $(1-\lambda)x + \lambda y$ 
instead of $W(x,y,\lambda)$. 

In the sequel, we denote a $W$-space simply by $X$. Let us define, for any $x,y\in X$, 
$[x,y]=\{(1-\lambda)x + \lambda y \mid \lambda\in[0,1]\}$.
A nonempty subset $C\subseteq X$  is said to be convex if for all $x,y\in C$, we have that $[x,y]\subseteq C$. 

Consider the following axioms:
$$
\begin{array}{ll}
\text{(W1)} & d(z,(1-\lambda)x + \lambda y) \le  (1-\lambda)d(z,x)+\lambda d(z,y),\\[1mm]
\text{(W2)} & d((1-\lambda)x + \lambda y,(1-\tilde{\lambda})x + \tilde{\lambda} y) =  \vert \lambda-\tilde{\lambda} \vert d(x,y), \\[1mm]
\text{(W3)} & (1-\lambda)x + \lambda y = \lambda y + (1-\lambda)x, \\[1mm]
\text{(W4)} & d((1-\lambda)x + \lambda z,(1-\lambda)y + \lambda w) \le (1-\lambda)d(x,y)+\lambda d(z,w),\\[1mm]
\text{(W5)} & 1x + 0 y=x \text{~and~} 0x + 1 y=y, \\[1mm]
\text{(W6)} & (1-\lambda)x + \lambda x=x, \\[1mm] 
\text{(W7)} &  d(x,(1-\lambda)x + \lambda y)\!=\!\lambda d(x,y) \text{~and~} d(y,(1-\lambda)x + \lambda y)\!=\!(1\!-\!\lambda)d(x,y).
\end{array}
$$

\noindent $W$-spaces satisfying (W1) were introduced by Takahashi \cite{Tak70} under the name of convex metric spaces.

\subsection{$W$-geodesic spaces}

\begin{definition}
A $W$-geodesic space is a  $W$-space $X$  satisfying $(W2)$ and $(W5)$. 
\end{definition}

Let $X$ be a $W$-geodesic space.

\begin{proposition}
(W6) and (W7) hold.
\end{proposition}
\begin{proof}
Use that, by (W5), $x=1x + 0 x=1x + 0y$, $y=0x + 1 y$, and apply (W2).
\end{proof}

Define, for any $x \ne y\in X$, the mapping
\begin{align*} 
W_{xy}:[0,d(x,y)]\to X,  &  \qquad W_{xy}(s)=\left(1-\frac{s}{d(x,y)}\right)x + \frac{s}{d(x,y)}y.
\end{align*}
We also define, for uniformity,  $W_{xx}:\{0\}\to X, \, \, W_{xx}(0)=x$.

\begin{proposition}
For all $x, y\in X$, $W_{xy}$ is a geodesic that joins $x$ and $y$ such that $W_{xy}([0,d(x,y)])=[x,y]$. 
Thus, $[x,y]$ is a geodesic segment that joins $x$ and $y$.
\end{proposition}
\begin{proof} The case $x=y$ is trivial, by (W6). Assume that $x\ne y$. 
One can easily see that, by (W2), $W_{xy}$ is a geodesic. Furthermore, by (W5), we have that 
$W_{xy}(0)=x$ and $W_{xy}(d(x,y))=y$. Thus, $W_{xy}$ is a geodesic that joins $x$ and $y$.
Since the mapping $[0,1]\to [0, d(x,y)], \, \lambda \mapsto \lambda d(x,y)$ is a bijection, it follows that 
$W_{xy}([0,d(x,y)])=[x,y]$.
\end{proof} 

In fact, a metric space is geodesic if and only if it is $W$-geodesic for some $W:X\times X\times [0,1]\to X$.

\begin{proposition}
For all $x, y\in X$ and all $\lambda \in [0,1]$, there exists a unique $z\in[x,y]$ 
(namely $z=(1-\lambda)x + \lambda y$) such that 
\begin{equation}
d(x,z)=\lambda d(x,y)\text{~and~}  d(y,z)=(1-\lambda)d(x,y). \label{W-geodesic-z-unique}
\end{equation}
\end{proposition}
\begin{proof}
Apply (W7), Lemma~\ref{geodesic-prop}.\eqref{geodesic-prop-1} and the previous proposition.
\end{proof}

\subsection{$W$-hyperbolic spaces}

A $W$-hyperbolic space \cite{Koh05} is a $W$-space satisfying (W1)-(W4). 
One can easily see that  (W5)-(W7) also hold in a $W$-hyperbolic space. 
In particular, any $W$-hyperbolic space is a $W$-geodesic space.

$W$-hyperbolic spaces turn out to be a natural class of geodesic spaces for the study of nonlinear 
iterations.  Normed spaces are obvious examples of $W$-hyperbolic spaces, as one can  define  
$W(x,y,\lambda)=(1-\lambda)x+\lambda y$. Busemann spaces \cite{Bus48,Pap05} and CAT(0) 
spaces \cite{AleKapPet19,BriHae99} are also $W$-hyperbolic spaces:
\begin{enumerate}
\item by \cite[Proposition~2.6]{AriLeuLop14}, Busemann spaces are the uniquely geodesic 
$W$-hyperbolic spaces;
\item by \cite[p. 386-388]{Koh08},  CAT(0) spaces are the $W$-hyperbolic spaces $X$  satisfying 
$$
d^2\left(z,\frac12 x + \frac12 y\right)\leq \frac12d^2(z,x)+\frac12d^2(z,y)-\frac14d^2(x,y) \quad \text{for all~} x,y,z\in X.
$$
\end{enumerate} 
It is well-known that any CAT(0) space is a Busemann space.

\section{The Tikhonov-Mann and modified Halpern iterations}

Let $X$ be a $W$-space, $C\subseteq X$ a convex subset, and $T:C\to C$ be a 
nonexpansive mapping, i.e. for all $x,y\in C$
\[ d(Tx,Ty)\le d(x,y). \]
 This very general setting suffices for defining the iterations of interest for
us in this paper. Let $(\beta_n)$ and $(\lambda_n)$ be sequences in $[0,1]$ and $u\in C$.

The \emph{Tikhonov-Mann iteration} $(x_n)$ and the \textit{modified Halpern iteration}  $(y_n)$ are defined as follows:
\begin{align}
x_0 \in C, \qquad & 
x_{n + 1}  = (1 - \lambda_n) u_n + \lambda_n T u_n,\label{def-TKM-W} \\
y_0\in C, \qquad & y_{n + 1}  = (1 - \beta_{n+1}) u + \beta_{n+1} v_n, \label{def-MH-W}
\end{align}  
where 
\begin{align}
u_n = (1 - \beta_n)u + \beta_nx_n \quad \text{and} \quad v_n = (1 - \lambda_n) y_n + \lambda_n T y_n. \label{def-un-vn}
\end{align}
The ordinary Halpern iteration is the special case of $(y_n)$ with 
$\lambda_n=1$ for all $n\in\N$ and was introduced (in the case where 
$u:=0$) by Halpern in \cite{Hal67}.
\begin{fact}\label{changed-HM}
We use for the parameter sequences  from the definition of the modified Halpern iteration different 
notations than the ones from 
\cite{KimXu05,CunPan11,KohSch12}: we write $1 - \beta_{n+1}$ instead of $\beta_n$ and $1-\lambda_n$ instead of $\alpha_n$. 
\end{fact}

We get from  \eqref{def-TKM-W}  the following inductive definition for $(u_n)$:
\begin{equation}\label{def-TKM-W-un-main}
\begin{aligned}
u_0 & = (1 - \beta_0) u + \beta_0 x_0 \in C, \\
u_{n + 1} & = (1 - \beta_{n + 1}) u + \beta_{n + 1}\big((1 - \lambda_n) u_n + \lambda_n T u_n\big).
\end{aligned}
\end{equation}

The following observation establishes the essential link between our iterations.

\begin{proposition}\label{main-link}
Assume that $y_0=(1-\beta_0) u + \beta_0 x_0$. Then for all $n\in \N$, 
\[u_n=y_n \text{~and~} x_{n+1}=v_n. \]
\end{proposition}
\begin{proof} 
We prove the first equality by induction on $n$. The case $n=0$ holds by hypothesis. As for the inductive step $n \Rightarrow n+1$:
\begin{align*}
y_{n + 1}  &\stackrel{\eqref{def-MH-W}}{=} (1 - \beta_{n+1}) u + \beta_{n+1} v_n 
\stackrel{\eqref{def-un-vn}}{=} (1 - \beta_{n + 1}) u + \beta_{n + 1} \big((1 - \lambda_n) y_n + \lambda_n T y_n\big) \\
&\stackrel{\rm (IH)}{=} (1 - \beta_{n + 1}) u + \beta_{n + 1} \big((1 - \lambda_n) u_n + \lambda_n T u_n\big) \\ 
      & \stackrel{\eqref{def-TKM-W-un-main}}{=} u_{n + 1}.
\end{align*} 
  
Furthermore, $x_{n+1}=(1 - \lambda_n) u_n + \lambda_n T u_n =(1 - \lambda_n) y_n + \lambda_n T y_n=v_n$. 
\end{proof}

\subsection{(Quantitative) conditions on $(\beta_n)$ and $(\lambda_n)$}

The following conditions on the sequences $(\beta_n)$ and $(\lambda_n)$ were used in the 
study of the asymptotic regularity and strong convergence of the 
 Tikhonov-Mann  and modified Halpern iterations:
\begin{align*}
(\cmh1) \quad  & \sum\limits_{n=0}^\infty (1 - \beta_n) =\infty, & 
(\cmh1^*) \quad  &  \prod\limits_{n=1}^\infty \beta_{n+1}=0, \\
(\cmh2) \quad  & \sum_{n = 0}^{\infty} \vert \beta_{n + 1} - \beta_n \vert < \infty, & 
(\cmh3) \quad  & \sum_{n = 0}^{\infty} \vert \lambda_{n + 1} - \lambda_n \vert<\infty,\\
(\cmh4) \quad  &  \lim\limits_{n \to \infty} \beta_n=1, & (\cmh5) \quad  & \liminf\limits_{n\to\infty} \lambda_n > 0, \\
(\cmh6) \quad & \lim\limits_{n \to \infty} \lambda_n = 1, & 
(\cmh7) \quad & \sum\limits_{n=0}^\infty (1-\lambda_n) =\infty. 
\end{align*}

Kim and Xu \cite{KimXu05} proved  the strong convergence of the modified Halpern iteration in uniformly smooth Banach spaces
under the hypotheses $(\cmh1)-(\cmh4)$, $(\cmh6)$, and $(\cmh7)$. Cuntavenapit 
and Panyanak \cite{CunPan11} showed that strong convergence holds in CAT(0) spaces without assuming 
$(\cmh7)$; they remarked that $(\cmh7)$ can be eliminated also  in the case of  Kim and Xu's result. Note that all the remaining conditions permit the choice 
of $\lambda_n=1$ (for all $n$) by which the modified Halpern iteration becomes 
the ordinary Halpern iteration. \\ It is easy to see that the proof in \cite{CunPan11}
can be modified in such a way that instead of $(\cmh6)$ only the weaker 
condition $(\cmh5)$ is needed: replace the last inequality in \cite[(3.4)]{CunPan11} by (using $(\cmh5)$ for $\lambda_n=1-\alpha_n$)
\[ (1-\alpha_n)d(x_n,Tx_n)\le d(x_n,x_{n+1})+\beta_nd(u,y_n)\to 0, \ \mbox{as} 
\ n\to\infty.\] 
Bo\c t, Csetnek and Meier \cite{BotCseMei19} used $(\cmh1)-(\cmh5)$ 
to obtain the strong convergence  of a modified Mann iteration in Hilbert spaces. By the comment above, Proposition \ref{main-link} and 
Lemma \ref{alpha-rconv-dxnun} below, this result follows from 
\cite{CunPan11}. \\
As a consequence of the quantitative results obtained by Cheval and 
Leu\c{s}tean \cite{CheLeu22}, conditions 
$(\cmh1)-(\cmh5)$ suffice for proving the asymptotic regularity of the Tikhonov-Mann 
iteration in $W$-hyperbolic spaces.  We also consider condition $(\cmh1^*)$ (which - for strictly positive $\beta_n>0$ - is equivalent with $(\cmh1)$), 
since, as observed for the first time by the second author \cite{Koh11}, its quantitative version is 
useful in obtaining better rates of asymptotic regularity.

As we are interested in effective bounds on the asymptotic behaviour of our iterations, we consider 
quantitative versions of the above conditions (with the exception  of $(\cmh7)$, which, as pointed above, is superfluous):\\[1mm]

\begin{tabular}{lll}
$(\cmh1_q)$ & $\sum\limits_{n=2}^\infty (1 - \beta_n)$ diverges with rate of divergence $\roneq$; \\[1mm]
$(\cmh1^*_q)$ & $\prod\limits_{n=1}^\infty \beta_{n+1}=0$ with rate of convergence $\ronestarq$; \\[1mm]
$(\cmh2_q)$ & $\sum\limits_{n=0}^\infty  \vert \beta_{n + 1} - \beta_n \vert$ converges with Cauchy modulus $\rtwoq$; \\[1mm]
$(\cmh3_q)$ & $\sum\limits_{n=0}^\infty  \vert \lambda_{n + 1} - \lambda_n \vert$ converges with Cauchy modulus $\rthreeq$; \\[2mm]
$(\cmh4_q)$ & $\lim\limits_{n \to \infty} \beta_n=1$ with rate of convergence $\rfourq$; \\[2mm]
$(\cmh5_q)$ & $\Lambda\in\N^*$ and $N_{\Lambda}\in\N$ are such that $\lambda_n \geq \frac1\Lambda$ 
for all $n\geq N_{\Lambda}$; \\[2mm]
$(\cmh6_q)$ & $\lim\limits_{n \to \infty} \lambda_n=1$ with rate of convergence $\rsixq$.
\end{tabular}

We refer, for example, to \cite{CheLeu22} for the definitions of quantitative notions such as rate of convergence, 
Cauchy modulus, rate of divergence.  The indices in our conditions above are chosen in such a way that the respective moduli 
satisfy the conditions in both \cite{CheLeu22} and \cite{KohSch12}.

\section{Rates of asymptotic regularity}\label{section-rates-as-reg}

Let us recall that if $X$ is a metric space, $\emptyset \ne C\subseteq X$, and $T:C\to C$, then a sequence $(a_n)$ in $C$ is said to be 
\begin{enumerate}
\item asymptotically regular if $\lim\limits_{n\to \infty} d(a_n,a_{n+1})=0$; a rate of asymptotic regularity of $(a_n)$ is 
a rate of convergence of  $(d(a_n,a_{n+1}))$ towards $0$.
\item $T$-asymptotically regular if $\lim\limits_{n\to \infty} d(a_n,Ta_n)=0$;  a rate of $T$-asymptotic regularity of $(a_n)$ is 
a rate of convergence of  $(d(a_n,Ta_n ))$ towards $0$.
\end{enumerate}

In the sequel, we explore the relation between rates of ($T$-)asymptotic regularity of the Tikhonov-Mann iteration
$(x_n)$ and those of the modified Halpern iteration $(y_n)$. 

For the rest of the section, $(X,d,W)$ is a $W$-hyperbolic space, $C$ is a convex subset of $X$, 
and $T:C\to C$ is a nonexpansive mapping.
We assume that $T$ has fixed points, hence the set $Fix(T)$ of fixed points of $T$ is nonempty. 
If $p$ is a fixed point of $T$, define 
\begin{equation}
M_p = \max\{d(x_0,p), d(u,p)\}. \label{def-M-bound-xn}
\end{equation}
By \cite[Lemma~3.1.(ii)]{CheLeu22} and \cite[Proposition 3.2.(8)]{CheLeu22}, we have that 
\begin{align}
d(x_n,u_n) \leq 2M_p(1-\beta_n). \label{dxn-un-Mp}
\end{align}
Let $K\in\N^*$ be such that $K\geq M_p$.

\begin{lemma}\label{alpha-rconv-dxnun}
Assume that $(\cmh4_q)$ holds. Then $\lim\limits_{n \to \infty} d(x_n,u_n)=0$ with rate of convergence 
\begin{equation}
\alpha(k) = \rfourq\left({2K(k + 1) - 1}\right). \label{def-alpha}
\end{equation}
\end{lemma} 
\begin{proof}
Let $n \geq \alpha(k)$. Then, by \eqref{dxn-un-Mp}, we get that 
\begin{align*}
d(x_n, u_n) &\leq 2M_p (1 - \beta_n) \leq 2K(1 - \beta_n)  \leq \frac{2K}{2K(k + 1)} = \frac{1}{k+1}.
\end{align*}
\end{proof}

\begin{proposition}\label{prop:asym-reg-MH-TM}
Assume that $(\cmh4_q)$ holds and let $\Phi:\N\to\N$. Define $\Phi':\N\to\N$ by 
\begin{equation}
\Phi'(k) := \max\left\{\alpha(3k+2), \Phi(3k+2) \right\},
\end{equation}
where $\alpha$ is given by \eqref{def-alpha}.
\begin{enumerate}
\item\label{xn-un-as-reg} If  $\Phi$ is a rate of ($T$-)asymptotic regularity of one of the sequences $(x_n)$, $(u_n)$, then 
$\Phi'$ is a rate of ($T$-)asymptotic regularity of the other one. 
\item\label{xn-yn-as-reg} Suppose, moreover, that $y_0 = (1 - \beta_0)u + \beta_0 x_0$. 
If one of the sequences $(x_n)$, $(y_n)$ is ($T$-)asymptotically regular with rate $\Phi$,
then  the other one is ($T$-)asymptotically regular with rate $\Phi'$. 
\end{enumerate}
\end{proposition}
\begin{proof}
\begin{enumerate}
\item Let $k \in \N$ and $n \geq \Phi'(k)$. 
Assume first that $\Phi$ is a rate of $T$-asymptotic regularity of $(u_n)$. 
We get that 
\begin{align*}
 d(x_n, T x_n) & \leq d(x_n, u_n) + d(u_n, T u_n) + d(Tu_n, Tx_n) \\
 & \leq 2d(x_n,u_n) + d(u_n, Tu_n) \quad [\text{since~} T \text{~is  nonexpansive}]\\ 
 & \leq 2d(x_n,u_n) + \frac{1}{3(k + 1)} \quad [\text{since~} n\geq \Phi(3k+2)] \\
 &\leq \frac{2}{3(k + 1)} + \frac{1}{3(k + 1)}=\frac1{k+1},
\end{align*} 
as $n\geq \alpha(3k+2)$, so we can apply  Lemma~\ref{alpha-rconv-dxnun}.

Assume now that $\Phi$ is a rate of asymptotic regularity of $(u_n)$.  Then
\begin{align*}
d(x_n,x_{n+1}) & \leq d(x_n,u_n)+d(u_n,u_{n+1})+d(u_{n+1},x_{n+1}) \\
& \leq d(x_n,u_n) + d(u_{n+1},x_{n+1}) + \frac{1}{3(k + 1)}\\
&\leq \frac{2}{3(k + 1)} + \frac{1}{3(k + 1)}=\frac1{k+1}.
\end{align*} 

The proof for the case when $\Phi$ is a rate of ($T$-)asymptotic regularity of $(x_n)$ follows by symmetry.

\item We have, by Proposition~\ref{main-link}, that $y_n=u_n$ for all $n\in\N$.  Apply \eqref{xn-un-as-reg}.
\end{enumerate}
\end{proof}

It follows  that if the starting points $x_0,y_0\in C$ satisfy $y_0 = (1 - \beta_0)u + \beta_0 x_0$, then
$(x_n)$ is ($T$-)asymptotically regular if and only if  $(y_n)$ is ($T$-)asymptotically regular.

\subsection{On rates of ($T$-)asymptotic regularity of the modified Halpern iteration}\label{rates-HM-subsection}

In \cite[Propositions~6.1,~6.2]{KohSch12}, Schade and the second author 
computed uniform rates 
of ($T$-)asymptotic regularity of the modified Halpern iteration in 
$W$-hyperbolic spaces. The hypotheses on the sequences 
$(\lambda_n)$,  $(\beta_n)$ used in \cite{KohSch12}
were $(\cmh1_q)$ (or - for strictly positive $\beta_n>0$ - equivalently, $(\cmh1^*_q)$), $(\cmh2_q)$ - $(\cmh4_q)$ 
and $(\cmh6_q)$. We improve  
these results by showing that the hypothesis $(\cmh6_q)$ can be weakened to $(\cmh5_q)$. 

Let $(y_n)$ be the modified Halpern iteration, given by \eqref{def-TKM-W}. Let $M\in\N^*$ be such that 
\begin{equation}
M\geq 4 \max\{d(u,p),d (y_0,p)\}
\end{equation}
for some $p\in Fix(T)$. 

The following lemma collects some properties of $(y_n)$ that will be useful in the sequel.

\begin{lemma} \cite{CunPan11,KohSch12} 
For all $n\geq 1$, 
\begin{align}
d(y_n,u) & \le M,  d(Ty_n,u)  \le M  \text{~and~}  d(y_n,y_{n+1}) \le M,\label{prop-5-2-6-8-9}\\
d(y_n,Ty_n) &  \leq d(y_n,y_{n+1})+(1-\beta_{n+1})d(u,Ty_n)+\beta_{n+1}(1-\lambda_n)d(y_n,Ty_n), \label{prop-5-1-6}\\
d(y_{n+1},y_n) & \le   \beta_{n+1}\big(d(y_n,y_{n-1}) \! + \! |\lambda_n \! - \! \lambda_{n-1}|d(y_{n-1},Ty_{n-1})\big)\!+\!|\beta_{n+1} \! - \! \beta_n|c_n, \label{prop-5-1-3}
\end{align}
where $c_n=(1-\lambda_{n-1})d(u,y_{n-1})+ \lambda_{n-1}d(u,Ty_{n-1})$.
\end{lemma}
\begin{proof} These properties are  proved in  \cite{KohSch12} following \cite{CunPan11} which treats
           the case of CAT(0) spaces. As we pointed out in  Remark~\ref{changed-HM}, the sequence $(x_n)$ and 
the scalars $\beta_n,\alpha_n$ in \cite{KohSch12} correspond to $(y_n)$ and $1-\beta_{n+1},1-\lambda_n$ in our current paper. Then
\eqref{prop-5-2-6-8-9} is \cite[Lemma 5.2(6),(8),(9)]{KohSch12},  \eqref{prop-5-1-6} is \cite[Lemma 5.1(6)]{KohSch12},   
and  \eqref{prop-5-1-3} is \cite[Proof of Lemma 5.1(3), last line on p.10]{KohSch12}.
\end{proof}

\begin{proposition}\label{HM-as-reg-new}
Assume that $(\cmh2_q)$, $(\cmh3_q)$ hold. Define \begin{equation} 
\gamma(k) = \max\left\{\rtwoq(8M(k + 1) - 1), \rthreeq(4M(k + 1) - 1) \right\}. \label{HM-def-gamma}
\end{equation}
\begin{enumerate}
\item\label{HM-as-reg-new-1} If $(\cmh1_q)$ holds, then $(y_n)$ is asymptotically regular with rate 
\begin{align}
\Sigma(k) &= \roneq\big(\gamma(k) + 
\left\lceil\ln(M(k + 1))\right\rceil + 1 \big) + 1. \label{HM-rate-ar-C1}
\end{align}
\item\label{HM-as-reg-new-1star} Suppose that $(\cmh1^*_q)$ holds, and $\psi: \N \to \N^*$ satisfies 
\begin{equation} 
\frac{1}{\psi(k)} \leq \prod\limits_{n = 0}^{\gamma(k)} \beta_{n + 1}. \label{HM-def-psi0}
\end{equation}
Then $(y_n)$ is asymptotically regular with rate
\begin{align}
\Sigma^*(k) &= \ronestarq\left(M\psi(k)(k + 1) - 1\right) + 1.
\label{HM-rate-ar-C1*} 
\end{align}
\end{enumerate}
\end{proposition}
\begin{proof}
This result is proven - in a different notation - in \cite[Propositions~6.1,~6.2]{KohSch12} 
observing that only the assumptions stated above are used there, where
\cite{KohSch12} in turn is based on \cite{KohLeu}. More 
specifically: 
\begin{enumerate}
\item replace the notations used for the modified Halpern iteration in 
\cite{KohSch12} with the ones from this paper (see Remark~\ref{changed-HM}).
\item use $\frac1{k+1}$ instead of $\varepsilon$.
\item replace $M_2$ with $M$, $\psi_\alpha$ with $\rthreeq$, $\psi_\beta$ with $\rtwoq$, and $\theta_\beta$ with $\roneq$ in (i) and with 
$\ronestarq$ 
in (ii), $D$ with $\frac{1}{\psi(k)}$ in (ii), with the corresponding changes in the parameters, 
due to the definitions, used in this paper, of the rates as mappings $\N\to\N$. 
\end{enumerate}
We get that  $\Sigma$ from \eqref{HM-as-reg-new-1} is $\tilde{\Phi}$ from \cite[Proposition~6.1]{KohSch12} with $\psi_\beta\left(\frac{\varepsilon}{8M_2}\right)$ replaced by  $\rtwoq(8M(k+1)-1)$, 
$\psi_\alpha\left(\frac{\varepsilon}{4M_2}\right)$ replaced by  $\rthreeq(4M(k+1)-1)$, 
and $\left\lceil \frac{M_2}{\varepsilon}\right\rceil$ replaced by $M(k+1)$. Furthermore, $\Sigma^*$ from \eqref{HM-as-reg-new-1star} is 
 $\tilde{\Phi}$ from \cite[Proposition~6.2]{KohSch12} with $\theta_\beta\left(\frac{D\varepsilon}{M_2}\right)$ replaced by $\ronestarq\left(M\psi(k)(k + 1) - 1\right)$. 
\end{proof}

\begin{proposition}\label{HM-as-reg-to-T-as-reg}
Assume that $(\cmh4_q)$ and $(\cmh5_q)$ hold. If $\Sigma: \N \to \N$ is a rate of asymptotic regularity of $(y_n)$,
then 
\begin{align}\label{def-wSigma}
\widehat{\Sigma}(k) & = \max\{N_\Lambda, \Sigma(2\Lambda (k + 1) - 1), \rfourq(2M\Lambda (k + 1) - 1)\}
\end{align}
is a rate of $T$-asymptotic regularity of $(y_n)$. 
\end{proposition}
\begin{proof} 
We get that for all $n\in\N$, 
\begin{align*}
d(y_n,Ty_n) &\stackrel{\eqref{prop-5-1-6}}{\leq}  d(y_n,y_{n+1})+(1-\beta_{n+1})d(u,Ty_n)+\beta_{n+1}(1-\lambda_n)d(y_n,Ty_n) \\
& \stackrel{\eqref{prop-5-2-6-8-9}}{\leq} d(y_n,y_{n+1})+M(1-\beta_{n+1})+(1-\lambda_n)d(y_n,Ty_n).
\end{align*}
After moving $(1-\lambda_n)d(y_n,Ty_n)$ to the left-hand side, we get that, for all $n\in\N$, 
\begin{align}
\lambda_n d(y_n,Ty_n) & \leq d(y_n,y_{n+1})+M(1-\beta_{n+1}).  \label{as-reg-HM-ineq-1}
\end{align}
Let now $n\geq \widehat{\Sigma}(k)$. Since $n\geq N_\Lambda$, we can apply $(\cmh5_q)$ to obtain that
$\lambda_n \geq \frac1\Lambda$. It follows from \eqref{as-reg-HM-ineq-1} that 
\begin{align*}
\frac1\Lambda d(y_n,Ty_n) & \leq d(y_n,y_{n+1})+M(1-\beta_{n+1}),  
\end{align*}
hence
\begin{align}
d(y_n,Ty_n) & \leq \Lambda d(y_n,y_{n+1})+M\Lambda(1-\beta_{n+1}).  \label{as-reg-HM-ineq-2}
\end{align}
As $n\geq \Sigma(2\Lambda (k + 1) - 1)$, we have that 
\begin{align}
d(y_n,y_{n+1}) & \leq \frac1{2\Lambda (k + 1)}. \label{as-reg-HM-ineq-3}
\end{align}
Since $n\geq \rfourq(2M\Lambda (k + 1) - 1)$, we get that 
\begin{align}
1-\beta_{n+1} &\leq \frac1{2M\Lambda (k + 1)}. \label{as-reg-HM-ineq-4}
\end{align}

Apply \eqref{as-reg-HM-ineq-2}, \eqref{as-reg-HM-ineq-3} and \eqref{as-reg-HM-ineq-4} to conclude that 
\begin{align*}
d(y_n,Ty_n) & \leq \frac1{2(k + 1)} + \frac1{2(k + 1)} =\frac1{k+1}.
\end{align*}
\end{proof}

Thus, as an application  of Propositions~\ref{HM-as-reg-new}, \ref{HM-as-reg-to-T-as-reg}, one computes also rates of 
$T$-asymptotic regularity of the modified Halpern iteration.

\mbox{}

One can easily see that particularizing the rates obtained in Propositions~\ref{HM-as-reg-new}.(ii) 
and \ref{HM-as-reg-to-T-as-reg} to the scalars 
$\lambda_n=\lambda \in (0,1]$ and $\beta_n = 1-\frac1{n+1}$ yields to 
quadratic rates of ($T$-)asymptotic regularity. In \cite{SabSht} a linear rate of 
convergence is obtained for some other Halpern-type iteration in the 
normed case for $\beta_n = 1-\frac{2}{n}.$ We now show that we also obtain 
this for the modified Halpern iteration in W-hyperbolic spaces using 
 \cite[Lemma 3]{SabSht}:
 
\begin{lemma}[\cite{SabSht}]\label{lemma3SabSht} Let $L>0$ and 
$(a_n)$ be a sequence of non-negative real numbers with $a_1\le L$ 
such that for 
$b_n=\min\{ 2/n,1\}$ we have for all $n\ge 1$,
\[ a_{n+1}\le (1-b_{n+1})a_n+(b_n-b_{n+1})c_n, \]
where $(c_n)$ is a sequence of reals such that $c_n\le L$ for all $n\geq 1$.

Then $a_n\le 2L/n$ for all $n\ge 1$. 
\end{lemma}

\begin{proposition}\label{linear-rate-Halpern} 
Let $\lambda_n=\lambda \in (0,1]$, $\beta_1=0$ and $\beta_n=1-\frac{2}{n}$ for $n\ge 2$.  Then 
for all $n\ge 0$, 
\begin{align*}
d(y_{n+1},y_n) & \le \frac{2M}{n+1},  \\
d(y_n,Ty_n) & \le \frac{4M}{\lambda(n+1)},
\end{align*}
where $M\in\N^*$ is such that $M\ge 4\max\{ d(u,p),d(y_0,p)\}$. 
\end{proposition}
\begin{proof} 

Apply \eqref{prop-5-1-3} to get that for all $n\geq 1$, 
\[ d(y_{n+1},y_n)\le \beta_{n+1}d(y_n,y_{n-1})+|\beta_{n+1}-\beta_n|c_n, \]
where $c_n=(1-\lambda)d(u,y_{n-1})+ \lambda d(u,Ty_{n-1})$.  By \eqref{prop-5-2-6-8-9},  we have that 
\[ \max\{ d(y_1,y_0),c_n\}\le M. \]
As $\beta_n=1-\min\{ 2/n,1\}$ for all $n\geq 1$, we can apply 
Lemma~\ref{lemma3SabSht} with $a_n=d(y_n,y_{n-1})$ to obtain that for $n\ge 0$,
\[ d(y_{n+1},y_n)\le \frac{2M}{n+1} \]
and so by - the proof of - \eqref{as-reg-HM-ineq-2} above 
\[ d(y_n,Ty_n)\le \frac{2M}{\lambda (n+1)}+\frac{2M}{\lambda (n+1)} =
\frac{4M}{\lambda (n+1)}. \]
\end{proof}

\subsection{From modified Halpern iteration to Tikhonov-Mann iteration}\label{rates-HM-TM-subsection}

We derive rates of ($T$-)asymptotic regularity of the Tikhonov-Mann iteration  from the rates 
of the modified Halpern iterations 
computed in Subsection~\ref{rates-HM-subsection}.

Let $(x_n)$ be the Tikhonov-Mann iteration, defined by \eqref{def-TKM-W}, and  $K\in\N^*$ be such that $K\geq M_p$, where $p$ is a fixed point of $T$ 
and $M_p$ is given by \eqref{def-M-bound-xn}.

\begin{proposition}\label{prop:TM-from-HM}
Assume that $(\cmh2_q)$, $(\cmh3_q)$, and $(\cmh4_q)$ hold.  
\begin{enumerate}
\item If $(\cmh1_q)$ holds and $\Sigma$ is defined as in Proposition~\ref{HM-as-reg-new}.\eqref{HM-as-reg-new-1}, then 
$(x_n)$ is asymptotically regular with rate 
\begin{align*}
\Phi(k) & = \max\{\rfourq(6K(k + 1) - 1), \Sigma(3k+2)\}.
\end{align*}
\item  If  $(\cmh1^*_q)$ holds and  $\psi$, $\Sigma^*$ are  as in Proposition~\ref{HM-as-reg-new}.\eqref{HM-as-reg-new-1star}, then 
$(x_n)$ is asymptotically regular with rate 
\begin{align*}
\Phi(k) & = \max\{\rfourq(6K(k + 1) - 1), \Sigma^*(3k+2)\}.
\end{align*}
\item If  $(\cmh5_q)$ holds and $\Sigma$ is a rate of asymptotic regularity of $(y_n),$ 
then
\begin{align*}
\widehat{\Phi}(k) & = \max\left\{\rfourq(6K(k + 1) - 1), N_\Lambda, \Sigma(6\Lambda(k + 1) - 1), 
\rfourq(24K\Lambda(k + 1) - 1) \right\}   
\end{align*}
is a rate of $T$-asymptotic regularity of $(x_n)$.
\end{enumerate}
\end{proposition}
\begin{proof}
Define $y_0 = (1 - \beta_0) u + \beta_0 x_0$ and consider the modified Halpern iteration 
$(y_n)$ starting with $y_0$. By an application of (W1), we get that 
$$d(y_0,p)\leq (1 - \beta_0) d(u,p) + \beta_0 d(x_0,p)\leq K.$$
 Hence, we can use Proposition~\ref{HM-as-reg-new} 
with $M=4K$ to get rates of asymptotic regularity of
$(y_n)$. Apply now Proposition~\ref{prop:asym-reg-MH-TM}.\eqref{xn-yn-as-reg} to obtain (i) and (ii).
\\ As for item (iii), assume, furthermore, that $(\cmh5_q)$ holds and let 
$\Sigma$ be a rate of asymptotic regularity of $(y_n)$. 
Then $\widehat{\Sigma}$ defined as in Proposition~\ref{HM-as-reg-to-T-as-reg} is a rate of $T$-asymptotic regularity of $(y_n)$.  
Applying Proposition~\ref{prop:asym-reg-MH-TM}.\eqref{xn-yn-as-reg}, we get that 
\begin{align*}
\widehat{\Phi}(k) & = \max\left\{\rfourq(6K(k + 1) - 1), N_\Lambda, \Sigma(6\Lambda(k + 1) - 1), 
\rfourq(24K\Lambda(k + 1) - 1) \right\}   
\end{align*}
is a rate of $T$-asymptotic regularity of $(x_n)$.
\end{proof}

We consider now again the case $\lambda_n=\lambda \in (0,1]$, $\beta_1=0$ and $\beta_n = 1-\frac{2}{n}$ for $n \ge 2$. 
Let $y_0=(1-\beta_0)u+\beta_0 x_0$. As $y_n=u_n$ (by Proposition~\ref{main-link}), we get from \eqref{dxn-un-Mp} that 
for all $n\ge 1$,
\[ d(x_n,y_n)\le \frac{4K}{n}.\]
Since $d(y_0,p)\le K$, we can apply Proposition~\ref{linear-rate-Halpern} with $M=4K$ and reason as  in the proof of 
Proposition \ref{prop:asym-reg-MH-TM}.\eqref{xn-un-as-reg} to obtain that for all $n\geq 1$, 
\begin{align*}
d(x_n,x_{n+1}) & \le d(x_n,y_n)+d(x_{n+1},y_{n+1})+d(y_n,y_{n+1})
\le \frac{4K}{n}+\frac{12K}{n+1} < \frac{16K}{n},\\
d(x_n,Tx_n) & \le 2d(x_n,y_n)+d(y_n,Ty_n)\le \frac{8K}{n}+
\frac{16K}{\lambda(n+1)} < \frac{24K}{\lambda n}.
\end{align*}
So we have obtained  a linear rate of 
asymptotic regularity also for the Tikhonov-Mann iteration as a consequence 
of the corresponding fact for the modified Halpern iteration and the 
reduction of the former to the latter. 

\subsection{From Tikhonov-Mann iteration to modified Halpern iteration}\label{rates-TM-HM-subsection}

Let  $(y_n)$ be the modified Halpern iteration, defined by \eqref{def-MH-W}, $p$ be a fixed point of $T$,
and $K\in\N^*$ be such that $K\geq \max\{d(u,p), d(y_0,p)\}$.

The following proposition gives rates of ($T$-)asymptotic regularity of $(y_n)$.

\begin{proposition}\label{prop:MH-from-TM}
Assume that $(\cmh2_q)$, $(\cmh3_q)$, and $(\cmh4_q)$ hold.  Define 
\begin{align*}
\chi(k) & = \max\{\rtwoq(8K(k + 1) - 1), \rthreeq(8K(k + 1) - 1)\},\\
\theta(k) &= \rfourq(6K(k + 1) - 1).
\end{align*}
\begin{enumerate}
\item\label{prop:MH-from-TM-1} Suppose that $(\cmh1_q)$ holds. Then 
\begin{enumerate}
\item $(y_n)$ is asymptotically regular with rate
\begin{align*}
\Sigma(k) & = \max\{\theta(k), \roneq(\chi(9k + 8) + 2 + \lceil\ln(18K(k + 1))\rceil) + 1 \}.
\end{align*}
\item If $(\cmh5_q)$ holds, then $(y_n)$ is $T$-asymptotically regular with rate 
\begin{align*}
\widehat{\Sigma}(k) & = \max\{\theta(k), N_\Lambda, \Sigma(6\Lambda(k + 1) - 1), \rfourq(12K\Lambda(k + 1) - 1) \}.\\
\end{align*}
\end{enumerate}
\item\label{prop:MH-from-TM-2} Suppose that $(\cmh1^*_q)$ holds and that $\psi: \N \to \N^*$ is such 
that $\frac{1}{\psi(k)} \leq \prod\limits_{n = 0}^{\chi(3k + 2)} \beta_{n + 1}.$
Then 
\begin{enumerate}
\item $(y_n)$ is asymptotically regular with rate
\begin{align*}
\Sigma^*(k) & =   \max\{\theta(k), \ronestarq(\psi^*(k)- 1) + 1,  \chi(9k+8) + 2\},
\end{align*}
where $\psi^*(k)=18K(k + 1)\psi(3k + 2)$. 
\item If $(\cmh5_q)$ holds, then $(y_n)$ is $T$-asymptotically regular with rate 
\begin{align*}
\widehat{\Sigma^*}(k) & = \max\{\theta(k), N_\Lambda, \Sigma^*(6\Lambda(k + 1) - 1), \rfourq(12K\Lambda(k + 1) - 1) \}.
\end{align*}
\end{enumerate}
\end{enumerate}
\end{proposition}
\begin{proof}
Take  $\beta_0=1$. Then, by (W5), $y_0=(1 - \beta_0) u + \beta_0 y_0$. 
Apply \cite[Theorems~4.1,4.2]{CheLeu22} for the Tikhonov-Mann iteration $(x_n)$ starting with $y_0$ to obtain 
rates of ($T$-)asymptotic regularity for this iteration and use
Proposition~\ref{prop:asym-reg-MH-TM}.\eqref{xn-yn-as-reg} to translate them into rates for $(y_n)$. 

For the proof of \eqref{prop:MH-from-TM-2}, remark that if $(\cmh1^*_q)$ holds, then $\ronestarq$ is also a rate of convergence of  
$\left(\prod\limits_{n=0}^\infty \beta_{n+1}\right)$ towards $0$, hence \cite[$(C2_q)$]{CheLeu22} holds with $\sigma_2:=\ronestarq$.
\end{proof}

\mbox{}

The first and the third author computed, for the particular case $\lambda_n=\lambda \in (0,1]$ and $\beta_n = 1-\frac1{n+1}$, quadratic rates of ($T$-)asymptotic regularity for the Tikhonov-Mann iteration (see \cite[Corollary 4.3]{CheLeu22}).  We show in the sequel that we can  use Lemma~\ref{lemma3SabSht}  for this iteration, too, and obtain, as a consequence, 
linear rates of  ($T$-)asymptotic regularity by letting $\beta_n=1-\frac{2}{n}$.

\begin{proposition}\label{linear-rate-TM} 
Assume that $\lambda_n=\lambda \in (0,1]$, $\beta_1=0$ and $\beta_n=1-\frac{2}{n}$ for $n\ge 2$, and let $(x_n)$ be the Tikhonov-Mann iteration.  Then 
for all $n\ge 1$, 
\begin{align*}
   d(x_{n + 1}, x_{n}) & \leq \frac{4K}{n}, \\
   d(x_n, T x_n) &\leq  \frac{8K}{\lambda n}.
\end{align*}
\end{proposition}
\begin{proof}
Using \cite[Proposition~3.2(7)]{CheLeu22}, we get that for all $n\geq 1$, 
\begin{equation*}
d(x_{n + 2}, x_{n+1}) \leq \beta_{n+1} d(x_{n+1}, x_n) + 2K |\beta_{n+1} - \beta_n|  \leq   \beta_{n+1} d(x_{n+1}, x_n) + 2K.
\end{equation*}
Moreover, $d(x_1, x_0) \leq d(x_1, p) + d(p, x_0) \leq 2K$. Applying Lemma~\ref{lemma3SabSht} for $a_n = d(x_{n+1}, x_n)$, $b_n = 1 - \beta_n$, $c_n = 2K$, and $L=2K$, 
we get that for all $n\geq 1$, 
 \begin{align*}
   d(x_{n + 1}, x_{n}) \leq \frac{4K}n. 
  \end{align*}
We obtain, as in the proof of \cite[Proposition~5.5]{CheLeu22}, that for all $n\geq 1$,
  \begin{align*}
    d(x_n, T x_n) &\leq \frac{1}{\lambda} d(x_n, x_{n + 1}) + \frac{2K}{\lambda} (1 - \beta_n) \leq 
 \frac{4K}{\lambda n} + \frac{4K}{\lambda n} = \frac{8K}{\lambda n}.
  \end{align*}
\end{proof}
 
We argue now as in Section~\ref{rates-HM-TM-subsection} to get, from Proposition~\ref{linear-rate-TM},  linear rates for  the modified Halpern iteration $(y_n)$: for all $n\geq 1$, 
  \begin{align*}
  d(y_{n + 1}, y_n)  & \leq d(x_n, y_n) + d(x_{n + 1}, y_{n + 1}) + d(x_n, x_{n + 1}) \leq \frac{12K}n, \\
    d(y_n, T y_n) & \leq 2d(x_n, y_n) + d(x_n, T x_n)  \leq \frac{16K}{\lambda n}. 
  \end{align*}
So for $d(y_n,Ty_n)$ the detour through the Tikhonov-Mann iteration gives 
(almost) exactly the same rate as the direct approach in Proposition 
\ref{linear-rate-Halpern} while the latter gives the slighty better constant 
`$8$' in the rate for $d(y_{n+1},y_n).$

\section{Rates of metastability}\label{section-rates-metastability}

Recall that if $X$ is a metric space and $(a_n)$ is a sequence in $X$, a function $\Omega:\N\times \N^\N\to\N$ is a rate of metastability 
of $(a_n)$ if it satisfies the following: for all  $k\in\N$ and  all $g:\N \to \N$, there exists $N\leq \Omega(k,g)$ such that 
\begin{align*}\label{def-metastability}
\forall i,j \in [N, N + g(N)] \left(d(a_i, a_j) \leq \frac{1}{k + 1}\right).
\end{align*}
Noneffectively, the above statement of metastability is trivially equivalent 
to the Cauchy property of $(a_n)$ and hence to its convergence if $X$ 
is complete. Whereas 
there are no computable rates of convergence for the iterations we consider 
(as a consequence of \cite{Neu15}), effective rates of metastability 
can be extracted even from highly noneffective convergence proofs 
(as the one given in \cite{CunPan11} using Banach limits and hence 
the axiom of choice) by general tools from proof theory. See \cite{Koh08} 
as well as
the recent survey \cite{Kohlenbach(ICM)}, where also a short history 
of metastability is given which goes back to Kreisel's seminal work in 
the early 50's (\cite{Kreisel(51),Kreisel(52)}), while the term `metastability'
was coined by Tao \cite{Tao(07)} who in turn refers to Jennifer Chayes' 
concept of a `metastability principle'. 
\\[1mm]
The following result shows that, as in the case of ($T$-)asymptotic regularity, there is a strong relation 
between rates of metastability of the Tikhonov-Mann iteration $(x_n)$ and the ones of the modified Halpern iteration $(y_n)$. 
The setting is the same as in Section~\ref{section-rates-as-reg}.

\begin{proposition}\label{prop:metastability-MH-to-TM}
Assume that $(\cmh4_q)$ holds and let $\Omega:\N\times \N^\N\to\N$. Define $\Omega':\N \times \N^\N \to \N$ by 
\begin{align}
\Omega'(k,g) &= \widetilde{\Omega}(3k+2, g, \alpha(3k+2)),
\end{align}
where $\alpha$ is given by \eqref{def-alpha} and $\widetilde{\Omega}:\N\times \N^\N \times\N\to\N$ is defined  by
\[
\widetilde{\Omega}(k, g, q) = \Omega(k, g_q) + q, 
\]
with $g_q:\N\to \N$, $g_q(n)= g(n + q) + q$.
\begin{enumerate}
\item\label{xn-un-metastability} If  $\Omega$ is  a rate of metastability of one of the sequences $(x_n)$, $(u_n)$, 
then  $\Omega'$ is  a rate of metastability of the other one. 
\item\label{xn-yn-metastability} Suppose that $y_0 = (1 - \beta_0)u + \beta_0 x_0$. 
If one of the sequences $(x_n)$, $(y_n)$ is Cauchy with rate of metastability $\Omega$, then the other one 
 is Cauchy with rate of metastability $\Omega'$.
\end{enumerate}
\end{proposition}
\begin{proof}
\begin{enumerate}
\item Assume first that $\Omega$ is a rate of metastability of $(u_n)$.\\[1mm]
{\bf Claim:} For all $k \in \N$, $g : \N \to \N$, $q \in \N$, there exists $N\in \N$ such that 
\[
q \leq N \leq \widetilde{\Omega}(k, g, q)  \text{~and~} \forall i,j \in [N, N + g(N)] 
\left(d(u_i, u_j) \leq \frac{1}{k + 1}\right). 
\]
{\bf Proof of claim:} Let $k \in \N$, $g : \N \to \N$, $q \in \N$. As $\Omega$ is a rate 
of metastability of $(u_n)$, it follows that there exists $N_0 \leq \Omega(k, g_q)$ such that 
\[
\forall i, j \in [N_0, N_0 + g_q(N_0)] \left(d(u_i, u_j) \leq \frac{1}{k + 1}\right).
\]
Let $N := N_0 + q$. Then $q \leq N \leq \Omega(k, g_q)+q =\widetilde{\Omega}(k, g, q)$. 
Furthermore, $[N, N + g(N)] = [N_0 + q, N_0 + q + g(N_0 + q)] = [N_0 + q, N_0 + g_q(N_0)] \subseteq [N_0, N_0 + g_q(N_0)]$,
hence $d(u_i, u_j) \leq \frac{1}{k + 1}$ for all $i, j \in [N, N+g(N)]$. \hfill $\blacksquare$\\[2mm]
Let $k \in \N$ and $g : \N \to \N$. Apply the claim for $k := 3k+2$, $g$ and $q:=\alpha(3k+2)$ to get the existence of
$N\in\N$ such that $\alpha(3k+2)\leq N\leq \widetilde{\Omega}(3k+2, g, \alpha(3k+2))=\Omega'(k, g)$ such that
\begin{equation}\label{meta-ineq-1}
\forall i, j \in [N, N + g(N)] \left(d(u_i, u_j) \leq  \frac{1}{3(k + 1)} \right). 
\end{equation}
It follows  that for all $i,j \in [N, N + g(N)]$, 
\begin{align*}
d(x_i, x_j) &\leq d(x_i, u_i) + d(u_i, u_j) + d(u_j, x_j) \\
& \leq \frac{1}{3(k + 1)} +  d(x_i, u_i) + d(u_j, x_j)  
\quad \text{by \eqref{meta-ineq-1}}\\
&\leq \frac{1}{3(k + 1)} + \frac{1}{3(k + 1)} + \frac{1}{3(k + 1)} = \frac{1}{k + 1},
\end{align*}
as $i,j\geq N\geq  \alpha(3k+2)$, and $\alpha$ is a rate of convergence towards $0$ of 
$(d(x_n,u_n))$, by Lemma~\ref{alpha-rconv-dxnun}.  

The proof for the case when $\Omega$ is a rate of metastability of $(x_n)$ is similar.
\item Apply Proposition~\ref{main-link} and \eqref{xn-un-metastability}.
\end{enumerate}
\end{proof}

The main results of \cite{KohSch12} are quantitative versions of the strong convergence, 
proved in \cite{CunPan11}, of the modified Halpern iteration $(y_n)$ in complete CAT(0) spaces. 
These quantitative versions provide effective uniform rates of metastability for $(y_n)$ (see \cite[Theorems~4.1, 4.2]{KohSch12} and also note \cite{KohLeu-Addendum} for a numerical improvement).
In Subsection~\ref{rates-HM-subsection} we improved  the quantitative results on the asymptotic 
regularity of $(y_n)$ obtained in \cite[Propositions~6.1,~6.2]{KohSch12} by weakening the hypothesis $(\cmh6_q)$ to $(\cmh5_q)$. 
One can easily see that this eliminates the hypothesis  $(\cmh6_q)$ in favor of  $(\cmh5_q)$  also
in  \cite[Theorems~4.1, 4.2]{KohSch12}, as it is not used in their 
proofs except via the rate of asymptotic regularity.
Hence, for CAT(0)  spaces,  new rates of metastability for $(y_n)$ are obtained by  considering the ones from \cite{KohSch12} with the new 
rates of ($T$-)asymptotic regularity computed in Subsection~\ref{rates-HM-subsection}. 
By Proposition~\ref{prop:metastability-MH-to-TM}.\eqref{xn-yn-metastability}, it follows that we can compute rates of 
metastability for the Tikhonov-Mann iteration $(x_n)$ in CAT(0) spaces, assuming that $(\cmh1_q)$ (or, equivalently for $\beta_n>0$,  $(\cmh1^*_q)$) and $(\cmh2_q)$-$(\cmh5_q)$ hold.

\section{Conclusions}

In this paper we showed that there is a strong relation between the modified Halpern iteration and the Tikhonov-Mann iteration for nonexpansive mappings. Thus, asymptotic regularity and strong convergence results  can be translated from one iteration to the other. This translation  holds also for quantitative versions of these results, providing rates of asymptotic regularity and rates of metastability. 
A future direction of research is to explore similar connections to other modified versions of the Halpern and Mann iterations.  One candidate  is the alternating Halpern-Mann iteration, introduced recently by Dinis and Pinto \cite{DinPin22}. 

By applying a lemma on sequences of real numbers due to Sabach and Shtern \cite[Lemma 3]{SabSht}, we obtained, for a particular choice of the scalars, linear rates of asymptotic regularity for the modified Halpern and Tikhonov-Mann iterations. Previous results guaranteed only quadratic such rates.  
This lemma was applied for the first time in \cite{SabSht} to get  linear rates of asymptotic regularity for the sequential averaging method (SAM), developed in \cite{Xu04}. As a consequence, one gets linear rates for the Halpern iteration also obtained (with an optimal constant) in the 
case of Hilbert spaces using a different technique in \cite{Lieder21}. 
Leu\c{s}tean and Pinto \cite{LeuPin22} computed, using  the same method,  linear rates for the alternating Halpern-Mann iteration.
Recently, the first and the third author in \cite{CheLeu23} applied (a version of) \cite[Lemma 3]{SabSht}  to  other classes of nonlinear iterations and, as a result, obtain linear rates of asymptotic regularity for these iterations.
  
 \mbox{}

{\bf Acknowledgement:} The second author has been supported by the 
German Science Foundation (DFG Project KO 1737/6-2).

\end{document}